\documentclass[reqno]{amsart}

\usepackage{general,resistance,pxfonts,leading}
\usepackage{cancel,xy}
\usepackage[switch*,pagewise]{lineno}

\usepackage[bookmarks, colorlinks=true, pdfstartview=FitV, linkcolor=blue, citecolor=blue, urlcolor=blue]{hyperref}
\usepackage{cite}  

\renewcommand{\llbracket}{\lbrack \negsp[3] \lbrack}  
\renewcommand{\rrbracket}{\rbrack \negsp[3] \rbrack}  

\numberwithin{equation}{section} \numberwithin{theorem}{section}

\leading{14pt}

\begin{document}
  \linenumbers

\title[Multiplication operators on the energy space]
  {Multiplication operators on the energy space}

\author{Palle E. T. Jorgensen}\email{jorgen@math.uiowa.edu}
\address{University of Iowa, Iowa City, IA 52246-1419 USA}

\author{Erin P. J. Pearse}\email{ep@ou.edu}
\address{University of Oklahoma, Norman, OK 73019-0315 USA}

\thanks{The work of PETJ was partially supported by NSF grant DMS-0457581. The work of EPJP was partially supported by the University of Iowa Department of Mathematics NSF VIGRE grant DMS-0602242.}

\begin{abstract}
  This paper studies the ``energy space'' $\mathcal{H}_{\mathcal{E}}$ (the Hilbert space of functions of finite energy, aka the Dirichlet-finite functions) on an infinite network (weighted connected graph), from the point of view of the multiplication operators $M_f$ associated to functions $f$ on the network. We show that the multiplication operators $M_f$ are not Hermitian unless $f$ is constant, and compute the adjoint $M_f^\star$ in terms of a reproducing kernel for $\mathcal{H}_{\mathcal{E}}$.  
  A characterization of the bounded multiplication operators is given in terms of positive semidefinite functions, and we give some conditions on $f$ which ensure $M_f$ is bounded. Examples show that it is not sufficient that $f$ be bounded or have finite energy. Conditions for the boundedness of $M_f$ are also expressed in terms of the behavior of the simple random walk on the network. 
  
  We also consider the bounded elements of $\mathcal{H}_{\mathcal{E}}$ and the (possibly unbounded) multiplication operators corresponding to them. In a previous paper, the authors used functional integration to construct a type of boundary for infinite networks. The boundary is described here in terms of a certain subalgebra of these multiplication operators, and is shown to embed into the Gel'fand space of that subalgebra. In the case when the only harmonic functions of finite energy are constant, we show that the Gel'fand space is the 1-point compactification of the underlying network. 
\end{abstract}



  \date{\bf\today.}

\maketitle

\setcounter{tocdepth}{1}
{\small \tableofcontents}

\allowdisplaybreaks


\section{Introduction}
\label{sec:introduction}

In this paper, we study the functions on a network, and the corresponding algebra of multiplication operators. More precisely, we consider the Hilbert space of finite-energy functions on a network, how the multiplication operators act on them, and under what conditions these operators are bounded, Hermitian, or have other properties of interest. In Theorem~\ref{thm:M_f-bounded-iff-s_f-semidefinite}, we show that the multiplication operator corresponding to a function $f$ is bounded on \HE with $\|M\| \leq b$ if and only if 
  \linenopax
  \begin{align}\label{eqn:s_f-semidefinite-preview}
    s_f(x,y) := \left(b^2 - f(x) \cj{f(y)} \right) \la v_x,v_y\ra_\energy
  \end{align}
is a positive semidefinite function on $X \times X$ (see \eqref{eqn:psd} for the definition of a positive semidefinite function), where $\{v_x\}_{x \in X}$ is the reproducing kernel for the Hilbert space of finite-energy functions discussed in \cite{DGG}. In Theorem~\ref{thm:d_x-is-bounded}, we show that multiplication by a point mass gives a bounded operator, and that the bound is given in terms of the conductance of the network at $x$ and the resistance distance to $x$. (While one would expect such boundedness, it is a bit surprising that the proof is not trivial.) In Theorem~\ref{thm:boundedness-criterion} we give an equivalent condition to  \eqref{eqn:s_f-semidefinite-preview} which is expressed in terms of an explicit matrix computation, and in Theorem~\ref{thm:boundedness-sufficiency} we give a sufficient condition for \eqref{eqn:s_f-semidefinite-preview} to hold which is even easier to check.

Next, we study the bounded functions of finite energy, and the corresponding multiplication operators. In Theorem~\ref{thm:Gelfand-space-is-bdG}, we show that the boundary $\bd \Graph$ developed in \cite{bdG} (see also the expository paper \cite{RBIN}) embeds into the Gel'fand space (that is, the spectrum of a Banach algebra realized as a topological space) of the algebra of bounded harmonic functions of finite energy. In Theorem~\ref{thm:Harm=0-means-Gelfand-is-1pt}, we then see that the Gel'fand space is a 1-point compactification of \Graph (and the unitalization of the corresponding $C^\ad$-algebra) if the only harmonic functions of finite energy are constants.

While our main results in this paper concern infinite weighted graphs, such as arise in the study of Markov processes \cite{Woess,WoessII,Lyons,LevinPeresWilmer}, geometric group theory \cite{Woess}, percolation \cite{Lyons}, discrete harmonic analysis \cite{Woess,WoessII,Soardi94}, and electrical networks \cite{Lyons,LevinPeresWilmer,DoSn84,Soardi94},%
  \footnote{This is by no means a complete catalogue of the literature, but each of the references listed in this sentence provides an excellent and extensive list of further reading.}
  we will need to develop some results on matrix-order and its use in the study of operators on (infinite-dimensional) separable Hilbert spaces. Aside from their applications, we hope that our separate matrix/operator results may be of independent interest. See \cite{RANR} for relations to Markov processes and \cite{SRAMO} for relations to matrix representations of operators. To make our paper accessible to separate audiences, we have included details from one area which perhaps may not be familiar to readers from the other. 
      The literature dealing with analysis on infinite graphs is vast, and we do not attempt to cite all the subareas. The monograph \cite{OTERN} includes a more systematic treatment, but still slanted towards spectral theory and operators in Hilbert space. It also contains a more complete bibliography. Stressing the operator theory/algebra, and reproducing kernels, there are the papers \cite{AMV09, KaVi06, SVY04, Cho08}; random walk models \cite{Die10, BoGa08, WoessII, SW09} and references cited there; and quantum theory \cite{Sal10, SSD09, Del09, MaSe09}.  

\section{Basic terms and previous results}
\label{sec:Basic-terms-and-previous-results}

We now proceed to introduce the key notions used throughout this paper: resistance networks, the energy form \energy, the Laplace operator \Lap, and their elementary properties. 

\begin{defn}\label{def:ERN}
  A \emph{(resistance) network} is a connected graph $(\Graph,\cond)$, where \Graph is a graph with vertex set \verts, and \cond is the \emph{conductance function} which defines adjacency by $x \nbr y$ iff $c_{xy}>0$, for $x,y \in \verts $. We assume $\cond_{xy} = \cond_{yx} \in [0,\iy)$, and write $\cond(x) := \sum_{y \nbr x} \cond_{xy}$. We require $\cond(x) < \iy$, but $\cond(x)$ need not be a bounded function on \Graph. The notation \cond may be used to indicate the multiplication operator $(\cond v)(x) := \cond(x) v(x)$.
\end{defn}

In this definition, connected means simply that for any $x,y \in \verts $, there is a finite sequence $\{x_i\}_{i=0}^n$ with $x=x_0$, $y=x_n$, and $\cond_{x_{i-1} x_i} > 0$, $i=1,\dots,n$.  
We may assume there is at most one edge from $x$ to $y$, as two conductors $\cond^1_{xy}$ and $\cond^2_{xy}$ connected in parallel can be replaced by a single conductor with conductance $\cond_{xy} = \cond^1_{xy} + \cond^2_{xy}$. Also, we assume $\cond_{xx}=0$ so that no vertex has a loop. 

Since the edge data of $(\Graph,\cond)$ is carried by the conductance function, we will henceforth simplify notation and write $x \in \Graph$ to indicate that $x$ is a vertex. For any network, one can fix a reference vertex, which we shall denote by $o$ (for ``origin''). It will always be apparent that our calculations depend in no way on the choice of $o$.

\begin{defn}\label{def:graph-laplacian}
  The \emph{Laplacian} on \Graph is the linear difference operator which acts on a function $v:\Graph \to \bR$ by
  \linenopax
  \begin{equation}\label{eqn:Lap}
    (\Lap v)(x) :
    = \sum_{y \nbr x} \cond_{xy}(v(x)-v(y)).
  \end{equation}
  A function $v:\Graph \to \bR$ is \emph{harmonic} iff $\Lap v(x)=0$ for each $x \in \Graph$. 
\end{defn}

We have adopted the physics convention (so that the spectrum is nonnegative) and thus our Laplacian is the negative of the one commonly found in the PDE literature. The network Laplacian \eqref{eqn:Lap} should not be confused with the stochastically renormalized Laplace operator $\cond^{-1} \Lap$ which appears in the probability literature, or with the spectrally renormalized Laplace operator $\cond^{-1/2} \Lap \cond^{-1/2}$ which appears in the literature on spectral graph theory (e.g., \cite{Chung01}).

\begin{defn}\label{def:graph-energy}
  The \emph{energy} of functions $u,v:\Graph \to \bC$ is given by the (closed, bilinear) Dirichlet form
  \linenopax
  \begin{align}\label{eqn:def:energy-form}
    \energy(u,v)
    := \frac12 \sum_{x,y \in \Graph} \cond_{xy}(\cj{u}(x)-\cj{u}(y))(v(x)-v(y)),
  \end{align}
  with the energy of $u$ given by $\energy(u) := \energy(u,u)$.
  The \emph{domain} of the energy form is
  \linenopax
  \begin{equation}\label{eqn:def:energy-domain}
    \dom \energy = \{u:\Graph \to \bC \suth \energy(u)<\iy\}.
  \end{equation}
\end{defn}
Since $\cond_{xy}=\cond_{yx}$ and $\cond_{xy}=0$ for nonadjacent vertices, the initial factor of $\frac12$ in \eqref{eqn:def:energy-form} implies there is exactly one term in the sum for each edge in the network. 

\begin{remark}\label{rem:sum-convergence}
  To remove any ambiguity about the precise sense in which \eqref{eqn:def:energy-form} converges, note that $\energy(u)$ is a sum of nonnegative terms and hence converges iff it converges absolutely. Since the Schwarz inequality gives $\energy(u,v)^2 \leq \energy(u)\energy(v)$, it is clear that the sum in \eqref{eqn:def:energy-form} is well-defined whenever $u,v \in \dom \energy$.
\end{remark}

\subsection{The energy space \HE} 
\label{sec:The-energy-space}

The energy form \energy is sesquilinear and conjugate symmetric on $\dom \energy$ and would be an inner product if it were positive definite.   

\begin{defn}\label{def:H_energy}\label{def:The-energy-Hilbert-space}
  Let \one denote the constant function with value 1 and recall that $\ker \energy = \bC \one$. Then $\HE := \dom \energy / \bC \one$ is a Hilbert space with inner product and corresponding norm given by
  \linenopax
  \begin{equation}\label{eqn:energy-inner-product}
    \la u, v \ra_\energy := \energy(u,v)
    \q\text{and}\q
    \|u\|_\energy := \energy(u,u)^{1/2}.
  \end{equation}
  We call \HE the \emph{energy (Hilbert) space}. 
\end{defn}

\begin{defn}\label{def:vx}\label{def:energy-kernel}
  Let $v_x$ be defined to be the unique element of \HE for which
  \linenopax
  \begin{equation}\label{eqn:v_x}
    \la v_x, u\ra_\energy = u(x)-u(o),
    \qq \text{for every } u \in \HE.
  \end{equation}
  The collection $\{v_x\}_{x \in \Graph}$ forms a reproducing kernel for \HE (\cite[Cor.~2.7]{DGG}); we call it the \emph{energy kernel} and \eqref{eqn:v_x} shows its span is dense in \HE. 
\end{defn}

Note that $v_o$ corresponds to a constant function, since $\la v_o, u\ra_\energy = 0$ for every $u \in \HE$. Therefore, $v_o$ may often be safely ignored or omitted during calculations. 

\begin{defn}\label{def:dipole}
  A \emph{dipole} is any $v \in \HE$ satisfying the pointwise identity $\Lap v = \gd_x - \gd_y$ for some vertices $x,y \in \Graph$. One can check that $\Lap v_x = \gd_x - \gd_o$; cf. \cite[Lemma~2.13]{DGG}.
\end{defn}

\begin{defn}\label{def:Fin}
  For $v \in \HE$, one says that $v$ has \emph{finite support} iff there is a finite set $F \ci \Graph$ for which $v(x) = k \in \bC$ for all $x \notin F$, i.e., the set of functions of finite support in \HE is 
  \linenopax
  \begin{equation}\label{eqn:span(dx)}
    \spn\{\gd_x\} = \{u \in \dom \energy \suth u(x)=k \text{ for some $k$, for all but finitely many } x \in \Graph\},
  \end{equation}
  where $\gd_x$ is the Dirac mass at $x$, i.e., the element of \HE containing the characteristic function of the singleton $\{x\}$. It is immediate from \eqref{eqn:def:energy-form} that $\energy(\gd_x) = \cond(x)$, whence $\gd_x \in \HE$.
  Define \Fin to be the closure of $\spn\{\gd_x\}$ with respect to \energy. 
\end{defn}

\begin{defn}\label{def:Harm}
  The set of harmonic functions of finite energy is denoted
  \linenopax
  \begin{equation}\label{eqn:Harm}
    \Harm := \{v \in \HE \suth \Lap v(x) = 0, \text{ for all } x \in \Graph\}.
  \end{equation}
\end{defn}

\begin{lemma}[{\cite[2.11]{DGG}}]
  \label{thm:<delta_x,v>=Lapv(x)}
  For any $x \in \Graph$, one has $\la \gd_x, u \ra_\energy = \Lap u(x)$.
\end{lemma}

The following result follows easily from Lemma~\ref{thm:<delta_x,v>=Lapv(x)}; cf.~\cite[Thm.~2.15]{DGG}.

\begin{theorem}[Royden decomposition]
  \label{thm:HE=Fin+Harm}
  $\HE = \Fin \oplus \Harm$.
\end{theorem}

\begin{remark}\label{rem:dual-systems}
  By combining \eqref{eqn:v_x} and the conclusion of Lemma~\ref{thm:<delta_x,v>=Lapv(x)}, one can reconstruct the network $(\Graph,\cond)$ (or equivalently, the corresponding Laplacian) from the dual systems (i) $(\gd_x)_{x \in X}$ and (ii) $(v_x)_{x \in X}$. Indeed, from (ii), we obtain the (relative) reproducing kernel Hilbert space \HE and from (ii), we get an associated operator $(\Lap u)(x) = \la \gd_x,u\ra_\energy$ for $u\in \HE$. In other words, (i) reproduces \Lap.
\end{remark}

The following results will be useful in the sequel, especially in \S\ref{sec:Bounded-functions-of-finite-energy}. For further details, please see \cite{DGG,ERM,bdG,RBIN} and \cite{OTERN}.

\begin{lemma}[{\cite[Lem~2.23]{DGG}}]
  \label{thm:energy-kernel-is-real}
  Every $v_x$ is \bR-valued, with $v_x(y) - v_x(o) >0$ for all $y \neq o$. 
\end{lemma}

\begin{lemma}[{\cite[Lem~6.9]{bdG}}]
  \label{thm:monopoles-and-dipoles-are-bounded}
  Every $v_x$ is bounded. In particular, $\|v_x\|_\iy \leq R(x)$. 
\end{lemma}

\begin{lemma}[{\cite[Lem~6.8]{bdG}}]
  \label{thm:Pfin-preserves-boundedness}
  If $v \in \HE$ is bounded, then $\Pfin v$ is also bounded. 
\end{lemma}

\begin{defn}\label{def:R(x)}
  Denote the \emph{(free) effective resistance} between $x$ and $o$ by 
  \linenopax
  \begin{align}\label{eqn:R(x)}
    R(x) := R^F(x,o) = \energy(v_x) = v_x(x) - v_x(o).
  \end{align}
  This quantity represents the voltage drop measured when one unit of current is passed into the network at $x$ and removed at $o$, and the equalities in \eqref{eqn:R(x)} are proved in \cite{ERM} and elsewhere in the literature; see \cite{Lyons,Kig03} for different formulations. 
\end{defn}

\begin{defn}\label{def:p(x,y)}
  Let $p(x,y) := \frac{\cond_{xy}}{\cond(x)}$ so that $p(x,y)$ defines a random walk on the network, with transition probabilities weighted by the conductances. Then let
  \linenopax
  \begin{align}\label{eqn:Prob[x->y]}
    \prob[x \to y] := \prob_x(\gt_y < \gt_x^+) 
  \end{align}
  be the probability that the random walk started at $x$ reaches $y$ before returning to $x$. In \eqref{eqn:Prob[x->y]}, $\gt_z$ is the hitting time of the vertex $z$.
\end{defn}

\begin{cor}[{\cite[Cor.~3.13 and Cor.3.15]{ERM}}]
  \label{thm:c(x)R(x)=Prob[x->o]}
  For any $x \neq o$, one has
  \linenopax
  \begin{equation}\label{eqn:c(x)R(x)=Prob[x->o]}
    \prob[x \to o] = \frac1{\cond(x) R(x)}.
  \end{equation}
\end{cor}


\section{Bounded multiplication operators}
\label{sec:Bounded-multiplication-operators}
     
Henceforth, we will write $X=\Graph\less\{o\}$ for brevity. 
Throughout the following, we use \gx to denote coefficients indexed by the vertices and write $\gx_x := \gx(x)$. Thus, \gx may or may not be an element of \HE. In order to perform computations which include both \gx and $u \in \HE$, we make the standing convention to choose the representative of $u$ (which we also denote by $u$) for which 
  \linenopax
  \begin{align}\label{eqn:o-convention}
    u(o)=0.
  \end{align}
It should be noted that under this convention, \Fin is the \energy-closure of the class of functions on \Graph which are constant (but not necessarily $0$) outside of a finite set. Also, this convention allows \eqref{eqn:v_x} to be written as 
  \linenopax
  \begin{equation}\label{eqn:v_x-}
    \la v_x, u\ra_\energy = u(x),
    \qq \text{for every } u \in \HE.
  \end{equation}

\begin{defn}\label{def:psd}
  A function $s:X \times X \to \bC$ is called \emph{positive semidefinite} (psd) iff for every finite subset $F \ci X$, one has 
  \linenopax
  \begin{align}\label{eqn:psd}
    \sum_{x,y \in F} \cj{\gx_x}\gx_y s(x,y) \geq 0, 
  \end{align}
for every function $\gx:X \to \bC$.  
\end{defn}

We shall have occasion to use basic tools from the theory of matrix-order, that is, the usual ordering of finite Hermitian matrices:
    \linenopax
    \begin{align}\label{eqn:matrix-order}
      A \geq 0
      \q\iff\q
      \la \gx, A\gx\ra_{\ell^2}
      = \sum_{x,y \in F} \cj{\gx_x} \gx_y A_{xy} \geq 0,
      \qq\forall F \text{ finite, and } \forall \gx.
    \end{align}
     
\begin{remark}[The role of finite subsets of $X$]
  \label{rem:finite}
  Definition~\ref{def:psd} is a statement about all possible finite Hermitian submatrices of the matrix $A$ with entries $A_{xy} = s_f(x,y)$. Thus, we will have frequent occasion to use the notation $F$ to indicate a finite subset of $V$, and we write 
  \linenopax
  \begin{align}\label{eqn:ell(F)}
    \ell(F) = \{\gx:F \to \bC \suth F \text{ is a finite subset of }X\}.
  \end{align}
\end{remark}

The order of Hermitian matrices, or of Hermitian (or self-adjoint) operators in Hilbert space is central in both harmonic analysis and in the theory of $C^\ad$-algebras. The reader may find the following references helpful: \cite{KadisonRingroseI, Arveson:invitation-to-Cstar, Sti55, Ber90, HornJohnson}. The following two lemmas are standard and proofs may be found in the references just listed. 

\begin{lemma}[The square root lemma]
  \label{thm:square-root-lemma}
  For a (finite) matrix $A$, one has $A \geq 0$ if and only if there is some $B \geq 0$ such that $A = B^2$.
\end{lemma}

\begin{lemma}\label{thm:sandwich-order}
  Let $A$ and $B$ be finite matrices. Then with respect to the ordering \eqref{eqn:matrix-order}, 
  \linenopax
  \begin{align}\label{eqn:sandwich-order}
    B^\ad A B \leq \|B\|^2 \, A
  \end{align}
  where the norm is the operator norm. In particular, if $B$ is the matrix of an orthogonal projection ($B=B^\ad=B^2$), then $A-BAB \geq0$, that is,
  \linenopax
  \begin{align}\label{eqn:sandwich-principle}
    \la u, Au\ra_{} \geq \la u, BAB u\ra_{}, 
    \qq\forall u \in \ell(F).
  \end{align}
\end{lemma}

\begin{defn}\label{def:multiplier}
  For a function $f:X \to \bC$, we denote by $M_f$ the corresponding multiplication operator:
  \linenopax
  \begin{align}\label{eqn:multiplier}
    (M_f u)(x) := f(x) u(x), \qq \forall x \in X.
  \end{align}
  When context precludes confusion, we suppress the dependence on $f$ and just write $M$. The norm of $M$ is the usual operator norm
  \linenopax
  \begin{align}\label{eqn:multiplier-norm}
    \|M\| := \|M\|_{\HE \to \HE} = \sup\{\|fu\|_\energy \suth \|u\|_\energy \leq 1\}.
  \end{align}
\end{defn}

It is important to notice that multiplication operators are a little unusual in \HE. The following feature of \HE operator theory contrasts sharply with the more familiar Hilbert spaces of $L^2$ functions, where all \bR-valued functions define Hermitian multiplication operators.

\begin{remark}\label{rem:||M_f||-is-not-sup(f)}
  One might guess that the operator norm of $M_f$ is computed from the sup-norm
of $f$, but this is not the case. In \S\ref{sec:Examples}, we give an example of a bounded function $f:X\to\bC$ for which the $M_f$, as an operator in \HE, is unbounded.
\end{remark}

\begin{lemma}\label{thm:multiplication-not-hermitian}
  For $f:X \to \bC$, the multiplication operator $M = M_f$ is Hermitian if and only if $f$ is constant and \bR-valued (in which case $f = 0$ in \HE). 
  \begin{proof}
    Choose any representatives for $u, v \in \HE$. From the formula \eqref{eqn:def:energy-form},
    \linenopax
    \begin{align*}
      \la M u, v\ra_\energy
      &= \frac12 \sum_{x,y \in \Graph} \cond_{xy} (\cj{f(x)u(x)}v(x) - \cj{f(x)u(x)}v(y) - \cj{f(y)u(y)}v(x) + \cj{f(y)u(y)}v(y)).
    \end{align*}
    By comparison with the corresponding expression, this is equal to $\la u, M v\ra_\energy$ iff $(\cj{f(y)} - f(x))\cj{u(y)}v(x) = (f(y) - \cj{f(x)})\cj{u(x)}v(y)$ holds for all $x,y \in \Graph$. However, since we are free to vary $u$ and $v$, it must be the case that $f$ is constant and $f=\cj{f}$.
  \end{proof}
\end{lemma}

Since Lemma~\ref{thm:multiplication-not-hermitian} shows that the adjoint of a multiplication operator is not what one would expect, one immediately wonders what the adjoint is, and this is the subject of our next result. 

\begin{lemma}\label{thm:M*v=fv}
  Let $M^\ad$ be the adjoint of the multiplication operator $M = M_f$ with respect to the energy inner product \eqref{eqn:energy-inner-product}. Then the adjoint of $M$ is defined by its action on the energy kernel:
  \linenopax
  \begin{align}\label{eqn:M*}
    M^\ad v_x = \cj{f(x)} v_x, \qq \forall x \in X.
  \end{align}
  \begin{proof}
    Since the energy kernel is dense in \HE, it suffices to show $\la v_y, M^\ad v_x - \cj{f(x)} v_x\ra_\energy = 0$ for every $y \in X$. Using \eqref{eqn:v_x-} for the final step, we have
    \linenopax
    \begin{align*}
      \la v_y, M^\ad v_x\ra_\energy
      = \la M v_y, v_x\ra_\energy
      = \la f \cdot v_y, v_x\ra_\energy
      = \cj{f(x) v_y(x)},
    \end{align*}
    which proves \eqref{eqn:M*} because the $v_y$ are \bR-valued by Lemma~\ref{thm:energy-kernel-is-real}.
  \end{proof}
\end{lemma}

\begin{remark}\label{rem:adjoint-as-scalar-multiple}
  Note that $M^\ad$ multiplies $v_x$ by the scalar $\cj{f(x)}$, not the function $\cj{f}$.
\end{remark}

\begin{lemma}\label{thm:bounded-operator}
  If $L$ is an operator on a Hilbert space \sH, then the following are equivalent:
  \begin{enumerate}[(i)]
    \item $L:\sH \to \sH$ is bounded with $\|L\| \leq b$.
    \item $b^2 - L^\ad L \geq 0$.
    \item $b^2 - L L^\ad \geq 0$.
  \end{enumerate}
\end{lemma}
In Lemma~\ref{thm:bounded-operator}, $L \geq 0$ means $\la u, Lu\ra \geq0$ for all $u$ in some dense subset of \sH, and of course $b^2$ means $b^2 \id$. The only nontrivial part of the proof of Lemma~\ref{thm:bounded-operator} is (ii)$\iff$(iii), which uses polar decomposition; see \cite{KadisonRingroseI}, for example.

\begin{theorem}\label{thm:M_f-bounded-iff-s_f-semidefinite}
  $M = M_f$ is bounded on \HE with $\|M\|_{\HE \to \HE} \leq b$ if and only if 
  \linenopax
  \begin{align}\label{eqn:s_f-semidefinite}
    s_f(x,y) := \left(b^2 - f(x) \cj{f(y)} \right) \la v_x,v_y\ra_\energy
  \end{align}
  is a positive semidefinite function on $X \times X$.
\end{theorem}
\begin{proof}
  We will work with the dense linear subspace $\sV := \spn \{v_x\}_{x \in X}$ of the energy space (density of \sV in \sH is shown in \cite{DGG}). By Lemma~\ref{thm:bounded-operator}, the first hypothesis in the statement of Theorem~\ref{thm:M_f-bounded-iff-s_f-semidefinite} is equivalent to 
  \linenopax
  \begin{align}\label{eqn:equiv-to(i)}
    \la u, (b^2-MM^\ad) u\ra_\energy \geq 0, \qq\forall u \in \sV.
  \end{align}
  Since $u \in \sV$ means $u = \sum_{x \in F} \gx_x v_x$ for some finite set $F \ci X$, we can evaluate \eqref{eqn:equiv-to(i)}:
  \linenopax
  \begin{align*}
    \la u, (b^2-MM^\ad) u\ra_\energy
    &= \sum_{x,y \in F} \cj{\gx_x} {\gx_y} \la v_x, (b^2-MM^\ad) v_y\ra_\energy \\
    &= \sum_{x,y \in F} \cj{\gx_x} {\gx_y} \left(\la v_x, b^2v_y\ra_\energy - \la M^\ad v_x, M^\ad v_y\ra_\energy\right) \\
    &= \sum_{x,y \in F} \cj{\gx_x} {\gx_y} \left(b^2\la v_x, v_y\ra_\energy - \la \cj{f(x)} v_x, \cj{f(y)} v_y\ra_\energy\right) \\
    &= \sum_{x,y \in F} \cj{\gx_x} {\gx_y} \left(b^2 - f(x)\cj{f(y)}\right) \la v_x, v_y\ra_\energy,
  \end{align*}
  where Lemma~\ref{thm:M*v=fv} was used to obtain the third equality.
  In view of \eqref{eqn:psd}, it is now clear that \eqref{eqn:equiv-to(i)} 
  holds for every choice of coefficients \gx if and only if $s_f(x,y)$ as defined in \eqref{eqn:s_f-semidefinite} is a positive semidefinite function on $X \times X$.
\end{proof}

\begin{cor}\label{thm:multiplier-product}
  If $f_1$ and $f_2$ are functions on $X$ and $\|M_{f_i}\| \leq b_i < \iy$ for $i=1,2$, then 
  \linenopax
  \begin{align}\label{eqn:psd-product}
    s_{12}(x,y) := \left(b_1b_2 - (f_1f_2)(x)\cj{(f_1f_2)(y)}\right) \la v_x, v_y\ra_\energy
  \end{align}
  is a positive semidefinite function on $X \times X$, where $(f_1f_2)(x) := f_1(x) f_2(x)$.
  \begin{proof}
    Since $M_{f_1} M_{f_2} = M_{(f_1 f_2)}$, we get $\|M_{(f_1 f_2)}\| \leq b_1 b_2$. Now Lemma~\ref{thm:bounded-operator} gives \eqref{eqn:psd-product}.
  \end{proof}
\end{cor}

\begin{remark}\label{rem:difficult}
  It is rather difficult to prove \eqref{eqn:psd-product} from first principles.
\end{remark}


\section{Algebras of multiplication operators}
\label{sec:Algebras-of-multiplication-operators}

We continue to use $X := \Graph\less\{o\}$ in conjunction with convention \eqref{eqn:o-convention}, as discussed at the beginning of \S\ref{sec:Bounded-multiplication-operators}. We begin by considering the multiplication operators $M_x := M_{\gd_x}$, that is, the special case of multiplication operators corresponding to the function 
\linenopax
\begin{align}\label{eqn:f=d_x}
  f:= \gd_{x} = 
  \begin{cases} 
  1, & y=x, \\ 
  0, & y \neq x
  \end{cases}  
\end{align}

The following operator will be very useful throughout the sequel.%
\footnote{This operator was introduced in \cite{SRAMO}, as far as we know. However, given the breadth and depth of the literature in this area, it is quite possible that it has appeared previously (in some guise) in the literature on random walks, percolation, or resistance forms.}

\begin{defn}[$V$ and \bV]\label{def:V}
  Let 
  \linenopax
  \begin{align}\label{eqn:V}
    V_{xy} := \la v_x, v_y\ra_\energy,
  \end{align}
  and define the inner product
  \linenopax
  \begin{align}\label{eqn:V-prod}
    \la \gx,\gh \ra_\bV = \sum_{x,y \in X} \cj{\gx_x} \gh_y V_{xy}
  \end{align}
  and corresponding norm $\| \gx\|_\bV = \sqrt{\la \gx,\gx \ra_\bV}$. 
  Then we have a Hilbert space
  \linenopax
  \begin{align}\label{eqn:bV}
    \bV = \{(\gx_x)_{x \in X} \suth \| \gx\|_\bV < \iy\}.
  \end{align}
\end{defn}

\begin{remark}\label{rem:polarization}
  Since $V$ is psd, one has that $\sum_{x,y \in F} \cj{\gx_x} \gx_y V_{xy} \geq 0$ for every finite subset $F$ of $X$, and so $\la \gx,\gx\ra_\bV$ can be defined by  \eqref{eqn:V-prod} as the supremum of the finite sums over $F$. Then $\la \gx,\gh \ra_\bV$ is obtained by polarization. For an alternative justification/definition, see Remark~\ref{rem:sum-convergence}.
  
  Note also that \eqref{eqn:V} defines an self-adjoint operator $V$ with $\dom V^{1/2} = \bV$. 
  We will see in Lemma~\ref{thm:V-and-bV} that \bV is unitarily equivalent to \HE.
\end{remark}

Recall from Definition~\ref{def:R(x)} that $R(x)$ denotes the (free) effective resistance between $x$ and $o$, and note that $R(x) = v_x(x)$ under the convention \eqref{eqn:v_x-}. Recall also from Definition~\ref{def:p(x,y)} that $\prob[x \to o]$ denotes the probability that the random walk started at $x$ reaches $o$ before returning to $x$.

\begin{theorem}\label{thm:d_x-is-bounded}
  For any $x \in X$, the multiplication operator $M_x$ is bounded on \HE with 
  \linenopax
  \begin{align}\label{eqn:Mx-bound}
    \|M_x\| = \sqrt{\cond(x) R(x)} = \prob[x \to o]^{-1/2}.  
  \end{align}
  \begin{proof}
    Define an operator on \HE via
    \linenopax
    \begin{align}\label{eqn:diagonal-decomp}
      (D_f V D_{\bar f})_{xy} := f(x) V_{xy} \cj{f(y)},
      \qq \forall x,y \in X \times X,
    \end{align}
    where $D_f$ is the diagonal operator whose \nth[x] diagonal entry is $f(x)$.
    Consequently, 
    \linenopax
    \begin{align}\label{eqn:s_f-preDirac}
      s_f(x,y) = (1-f(x)\cj{f(y)})\la v_x, v_y\ra_\energy 
      = V_{xy} - (D_f V D_{\bar f})_{xy}
      = (V - D_f V D_{\bar f})_{xy}.
    \end{align}
    By Theorem~\ref{thm:M_f-bounded-iff-s_f-semidefinite}, we need to show that $s_{\gd_{x_o}}$ is psd, but $f=\gd_{x_o}$ changes \eqref{eqn:s_f-preDirac} into
    \linenopax
    \begin{align}\label{eqn:s_f-Dirac}
      s_{\gd_{x_o}}(x,y)
      = V_{xy} - V_{x_o,x_o} \gd_{(x,y),(x_o,x_o)},
    \end{align}
    where $\gd_{(x,y),(x_o,x_o)}$ is a Kronecker delta for matrix position $(x_o,x_o)$. 
    To check that \eqref{eqn:s_f-Dirac} is psd, suppose that $F \ci X$ is any finite subset containing $x_o$ so that positive semidefiniteness is equivalent to 
    \linenopax
    \begin{align}\label{eqn:sandwich}
      V - P_o V P_o \geq 0,
    \end{align}
    where $P_o$ is the projection in \bV onto the 1-dimensional subspace of $\{f:F \to \bC\}$ spanned by $\gd_{x_o}$. So the boundedness of $M_x$ follows from \eqref{eqn:sandwich-principle}. 
    
    It remains to compute the norm. First, note that for any $u \in \HE$, one immediately has $\|M_x u\|_\energy^2 = \cond(x) |u(x)|^2$ from \eqref{eqn:def:energy-form}. Then \eqref{eqn:o-convention}, the Schwarz inequality and \eqref{eqn:v_x} give
    \linenopax
    \begin{align*}
      |u(x)|
      = |u(x)-u(o)|
      = \left|\la v_x,u\ra_\energy\right|
      \leq \|v_x\|_\energy \|u\|_\energy
      = \sqrt{R(x)} \|u\|_\energy,
    \end{align*}    
    and \eqref{eqn:Mx-bound} follows upon multiplying across by $\sqrt{\cond(x)}$ and applying Corollary~\ref{thm:c(x)R(x)=Prob[x->o]}.
  \end{proof}
\end{theorem}

\subsection{The multiplier $C^\ad$-algebra of \HE}
\label{sec:multiplier-C*algebra-of-HE}

Our work in this section is inspired in part by work on quantum graphs as systems of coherent state configurations on countable graphs. See \cite{AlbeverioCebulla07, AlbeverioBozhokKoshmanenko06, AlbeverioOstrovskyiSamoilenko07, KostrykinSchrader06, KostrykinPotthoffSchrader07}, for example.

\begin{defn}\label{def:multiplier-C*algebra}
  Define the multiplier $C^\ad$-algebra of \HE to be the $C^\ad$-subalgebra of $\sB(\HE)$ generated by the bounded multiplication operators $M_f$. We denote this algebra by 
  \linenopax
  \begin{align}\label{eqn:C*(HE)}
     C^\ad(\HE) := \bigvee \{M_f, M_f^\ad \suth f:X \to \bC \text{ and } M_f \text{ is bounded}\},
  \end{align}
  and the relations defining this algebra are given in Corollary~\ref{thm:C(M_x)-generators-and-relations}. In \eqref{eqn:C*(HE)}, the symbol $\bigvee$ indicates that the linear span is closed in the operator topology; i.e., the uniform norm of bounded operators.
\end{defn}

\begin{remark}\label{rem:abelian}
  There is an important distinction between the abelian algebra generated by $M_f$ (with $f$ such that $s_f$ is psd, as in Theorem~\ref{eqn:s_f-semidefinite}), and the $C^\ad$-algebra generated by $M_f$. The first is abelian and the second very non-abelian.
\end{remark}

\begin{remark}\label{rem:d_x-is-bounded}
  Theorem~\ref{thm:d_x-is-bounded} shows that $M_x \in C^\ad(\HE)$, and hence that $M_f \in C^\ad(\HE)$ for every finitely supported function $f:X \to \bC$. 
\end{remark}

Recall that $|u\ra \la v|$ is Dirac's notation for the rank-1 operator that sends $v$ to $u$, and it is a projection if and only if both $u$ and $v$ are unit vectors.

\begin{lemma}\label{thm:M^ad-as-projn}
  For any $x \in X$, $M_x$ and $M_x^\ad$ are the rank-1 operators expressed in Dirac notation by 
  \linenopax
  \begin{align}\label{eqn:M^ad-as-projn}
    M_x = |\gd_x\ra \la v_x|
    \q\text{and}\q
    M_x^\ad = |v_x\ra \la \gd_x|. 
  \end{align}
  \begin{proof}
    It suffices to verify the second identity in \eqref{eqn:M^ad-as-projn} on the dense set $\spn\{v_x\}$:
    \linenopax
    \begin{align*}
      M_x^\ad v_y 
      = \gd_x(y) v_y 
      &= \left.\begin{cases}
           v_x, & y=x,\\
           0, &\text{else},
         \end{cases} \right\}
      = (\gd_y(x)-\gd_y(o)) v_x
      = v_x \la \gd_x, v_y \ra_\energy
      = |v_x\ra \la \gd_x| v_y,
    \end{align*}
    where we have used \eqref{eqn:v_x-}. Now the first identity in \eqref{eqn:M^ad-as-projn} follows from the second.
    
    For an alternative proof, note that $M_f^\ad v_x = \cj{f(x)} v_x$, by Lemma~\ref{thm:M*v=fv}, which implies that $M_x^\ad = |v_x \ra \la \gd_x|$. Then $M_x = (M_x^\ad)^\ad = |v_x \ra \la\gd_x|^\ad = |\gd_x \ra \la v_x|$.
  \end{proof}
\end{lemma}

Note that \eqref{eqn:def:energy-form} immediately gives
  \linenopax
  \begin{align}\label{eqn:<d_x,d_y>=cond}
    \la \gd_x, \gd_y \ra_\energy
    =\begin{cases}
      -\cond_{xy}, &x\neq y,\\
      \cond(x), &x=y.
    \end{cases}
  \end{align}

\begin{remark}\label{rem:d_x-is-bounded-from-M^ad-as-projn}
  One can prove Theorem~\ref{thm:d_x-is-bounded} from Lemma~\ref{thm:M^ad-as-projn}:
  \linenopax
  \begin{align}\label{eqn:d_x-is-bounded-from-M^ad-as-projn}
    \|M_x\|_{\HE \to \HE}
    &= \left\| |\gd_x\ra \vstr[2.1] \la v_x| \right\|_{\HE \to \HE}
    = \|\gd_x\|_\energy \|v_x\|_\energy
    = \sqrt{\cond(x)} \sqrt{R(x)}.
  \end{align}
\end{remark}

\begin{cor}\label{thm:C(M_x)-generators-and-relations}
  $C^\ad(\HE)$ is the $C^\ad$-subalgebra of $\sB(\HE)$ with generators $\{M_x,M_x^\ad\}_{x \in X}$ and relations
  \linenopax
  \begin{align}
    M_x^\ad M_y 
    &= \la \gd_x, \gd_y \ra_\energy |v_x \ra \la v_y|, \label{eqn:M_x*M_y} \\
    M_x M_y^\ad 
    &= \la v_x, v_y \ra_\energy |\gd_x \ra \la \gd_y|. \label{eqn:M_xM_y*}
  \end{align}
  where $\la \gd_x, \gd_y \ra_\energy$ is as in \eqref{eqn:<d_x,d_y>=cond}.
  \begin{proof}
    The computations are direct applications of \eqref{eqn:M^ad-as-projn} and Dirac's notation:
    \linenopax
    \begin{align*}
      M_x^\ad M_y 
      &= \left(|v_x \ra \vstr[2.2]\la \gd_x|\right)\left(|\gd_y \ra \la v_y|\right)
      = |v_x \ra \vstr[2.2]\la \gd_x|\gd_y \ra \la v_y|
      = |v_x \ra \la \gd_x, \gd_y \ra_\energy \la v_y|
      = \la \gd_x, \gd_y \ra_\energy |v_x \ra \la v_y|,
    \end{align*}
    and similarly for $M_x M_y^\ad$.
  \end{proof}
\end{cor}

\begin{remark}\label{rem:C*(HE)-is-big}
  Corollary~\ref{thm:C(M_x)-generators-and-relations} shows that $C^\ad(\HE)$ contains all the rank-1 projections corresponding to the functions $\{v_x\}$. Since the span of this set is dense in \HE, this implies that $C^\ad(\HE)$ contains all finite-rank operators, and hence all the compact operators (since the compact operators are obtained by closing the space of finite-rank operators). Thus Corollary~\ref{thm:C(M_x)-generators-and-relations} shows that $C^\ad(\HE)$ is quite large.
\end{remark}

\begin{remark}\label{rem:normalized}
  Let us introduce the normalized functions
  \linenopax
  \begin{align}\label{eqn:normalized-functions}
    u_x := \frac{v_x}{\|v_x\|_\energy}
    \q\text{and}\q
    d_x := \frac{\gd_x}{\|\gd_x\|_\energy}
  \end{align}
  and the corresponding rank-1 projections onto the spans of these elements:
  \linenopax
  \begin{align}\label{eqn:normalized-projections}
    U_x 
    &:= |u_x \ra \la u_x| 
    =\proj \spn u_x 
    = \frac1{\cond(x) R(x)} M_x^\ad M_x 
    = (\prob[x \to o]) M_x^\ad M_x \\
    \noalign{and}
    D_x 
    &:= |d_x \ra \la d_x|
    = \proj \spn d_x 
    = \frac1{\cond(x) R(x)} M_x M_x^\ad
    = (\prob[x \to o]) M_x M_x^\ad.
  \end{align}
  Then one has two systems of orthonormal projections satisfying the relations
  \linenopax
  \begin{align*}
    U_x U_y &= \frac{\la v_x,v_y\ra_\energy}{\sqrt{R(x)R(y)}}|u_x\ra\la u_y|, &
    U_x D_y &= \la u_x,d_y\ra_\energy \; |u_x\ra\la d_y|, \\
    D_x U_y &= \la d_x,u_y\ra_\energy \; |d_x\ra\la u_y|, &
    D_x D_y &= \frac{\la \gd_x,\gd_y\ra_\energy}{\sqrt{\cond(x)\cond(y)}}|d_x\ra\la d_y|.
  \end{align*}
  Moreover, one also has
  \linenopax
  \begin{align}\label{eqn:initial-and-final}
    \bigvee_{x \in X} \ran U_x = \HE,
    \q\text{and}\q 
    \bigvee_{x \in X} \ran D_x = \Fin,
  \end{align}
  where $\bigvee$ indicates that one takes the closed linear span.
\end{remark}

Theorem~\ref{thm:boundedness-criterion} gives a necessary and sufficient condition for determining whether or not an operator is bounded. In the statement and proof, the ordering is as defined by \eqref{eqn:matrix-order}. It will also be helpful to keep in mind that \\
\scalebox{0.85}{
\begin{minipage}{5in}
\begin{align}
  D_f V D_{\cj{f}} &=
  \left[\begin{array}{ccccc}
    f(x_1) & 0 & 0 & \dots \\
    0 & f(x_2) & 0 & \dots \\
    0 & 0 & f(x_3) &\dots \\
    \vdots & \vdots & & \ddots
  \end{array}\right]
  \left[\begin{array}{ccccc}
    V_{x_1 x_1} & V_{x_1 x_2} & V_{x_1 x_3} & \dots \\
    V_{x_2 x_1} & V_{x_2 x_2} & V_{x_2 x_3} & \dots \\
    V_{x_3 x_1} & V_{x_3 x_2} & V_{x_3 x_3} & \dots \\
    \vdots & \vdots & \vdots & \ddots
  \end{array}\right]
  \left[\begin{array}{ccccc}
    \cj{f(x_1)} & 0 & 0 & \dots \\
    0 & \cj{f(x_2)} & 0 & \dots \\
    0 & 0 & \cj{f(x_3)} &\dots \\
    \vdots & \vdots & & \ddots
  \end{array}\right] \notag
  \\
  &\hstr[10]=
  \left[\begin{array}{ccccc}
    f(x_1)V_{x_1 x_1}\cj{f(x_1)} & f(x_1)V_{x_1 x_2}\cj{f(x_2)} & f(x_1)V_{x_1 x_3}\cj{f(x_3)} & \dots \\
    f(x_2)V_{x_2 x_1}\cj{f(x_1)} & f(x_2)V_{x_2 x_2}\cj{f(x_2)} & f(x_2)V_{x_2 x_3}\cj{f(x_3)} & \dots \\
    f(x_3)V_{x_3 x_1}\cj{f(x_1)} & f(x_3)V_{x_3 x_2}\cj{f(x_2)} & f(x_3)V_{x_3 x_3}\cj{f(x_3)} & \dots \\
    \vdots & \vdots & \vdots & \ddots
  \end{array}\right],
  \label{eqn:diagonal-decomp-matrices}
\end{align}
\end{minipage}
}\\
as in \eqref{eqn:diagonal-decomp}, and that $V_F$ and $D_F$ are the \emph{finite} submatrices of $V$ and $D$ obtained by taking only the rows and columns corresponding to those vertices $x$ which lie in the finite subset $F \ci X$. The limit of the filter $\{T_F\}_{F \ci X}$ of the operators defined in \eqref{eqn:boundedness-criterion} will be computed in Corollary~\ref{thm:T_F-converges-to-M_f*}.

\begin{theorem}\label{thm:boundedness-criterion}
  The multiplication operator $M=M_f$ is bounded on \HE if and only if the family of operators 
  \linenopax
  \begin{align}\label{eqn:boundedness-criterion}
    T_F = V_F^{1/2} D_F V_F^{-1/2}
  \end{align}
  is uniformly bounded, i.e., there exists a constant $b < \iy$ such that $\sup_F \|T_F\|_{\bV \to \bV} \leq b$, as $F$ ranges over all finite subsets of $X$.
  Here $V_F$ is the truncation of \eqref{eqn:V} with entries $(V_{xy})_{x,y \in F}$, and $D_F$ is the truncated diagonal operator with entries $\left(f(x)\gd_{x,y}\right)_{x,y \in F}$. 
  
  In the case when these equivalent conditions are satisfied, 
  \linenopax
  \begin{align}\label{eqn:boundedness-criterion-bound}
    \|M_f\|_{\HE \to \HE} = \sup_F \|T_F\|_{\bV \to \bV} \leq b,
  \end{align}
  where the sup is taken over all finite subsets $F \ci X$.
  \begin{proof}
    From Theorem~\ref{thm:M_f-bounded-iff-s_f-semidefinite}, we know that $M$ is bounded iff $s_f(x,y)$ in \eqref{eqn:s_f-semidefinite} is semidefinite, and this inequality can be written in terms of matrices as $b^2 V - DVD \geq 0$, with respect to the ordering \eqref{eqn:matrix-order}; see Lemma~\ref{thm:sandwich-order}. This transforms a difficult condition (positive semidefiniteness) into an easier condition to check:
    \linenopax
    \begin{align}\label{eqn:boundedness-criterion-der1}
      b^2 \la \gx,V_F\gx \ra_{\bV} - \la \gx,D_F V_F D_F\gx \ra_{\bV} \geq 0,
      \qq \forall \gx \in \bV.
    \end{align}  
    Note that $V$ is psd (essentially by definition):
    \linenopax
    \begin{align*}
      \sum_{x,y \in F} \cj{\gx_x} \gx_y V_{x,y} 
      = \sum_{x,y \in F} \cj{\gx_x} \gx_y \la v_x, v_y\ra_\energy 
      = \left\la \sum_{x \in F} \gx_x v_x, \sum_{y \in F} \gx_y v_y \right\ra_\energy 
      = \left\|\sum_{x \in F} \gx_x v_x \right\|_\energy^2 
      \geq 0, 
    \end{align*}
    and so we have $V = (V^{1/2})^2$ by Lemma~\ref{thm:square-root-lemma}. 
    Then \eqref{eqn:boundedness-criterion-der1} gives
    \linenopax
    \begin{align}\label{eqn:T_F-converges-to-M_f*-der2}
      \| V_F^{1/2} D_F \gx \|_\bV^2
      =\| T_F V_F^{1/2} \gx\|_\bV^2 
      \leq b^2 \|V_F^{1/2} \gx\|_\bV^2
      \qq \forall \gx \in \bV.
    \end{align}
    Thus there is a bounded operator sending \smash{$V_F^{1/2}\gx$} to \smash{$V_F^{1/2} D_F \gx$}, for any $\gx \in \bV$. Less grandiosely, this means there is an $n \times n$ matrix $T_F$ satisfying
    \linenopax
    \begin{align}\label{eqn:boundedness-criterion-der3}
      T_F V_F^{1/2} = V_F^{1/2} D_F,
      \qq\text{and}\qq
      \|T_F\|_{\bV \to \bV} \leq b.
    \end{align}
    From \eqref{eqn:boundedness-criterion-der3}, it is clear that $T_F$ is given by \eqref{eqn:boundedness-criterion}, and the independence of $b$ from $F$ follows by the Uniform Boundedness principle.  
  \end{proof}
\end{theorem}

\begin{remark}\label{rem:T-not-self-adjoint}
  Note that $T_F$ is not self-adjoint for general finite $F$ (even in the case when $f$ is \bR-valued) because
  \linenopax
  \begin{align*}
    \left(V_F^{1/2} D_F V_F^{-1/2}\right)^\ad = V_F^{-1/2} D_F V_F^{1/2}.
  \end{align*}
  However, one can still compute the operator norm of $T_F$ as the square root of the largest eigenvalue of $T_F^\ad T_F$.
\end{remark}

\begin{remark}\label{rem:hardness-of-boundedness-criterion}
  Even in the case when $M_z = M_{\gd_z}$, it may be very difficult to use \eqref{eqn:boundedness-criterion} to compute $\|M_z\|$, and preferable to use Theorem~\ref{thm:d_x-is-bounded} instead. In this situation, one has only
  \linenopax
  \begin{align*}
    T_F = V_F^{1/2} (\gd_{x,z}\gd_{y,z}) V_F^{-1/2},
  \end{align*}
  but it is even difficult to compute the entries of $V_F^{1/2}$ and $V_F^{-1/2}$.
\end{remark}

Our next goal is to compute the limit of the filter $\{T_F\}_{F \ci X}$ in Corollary~\ref{thm:T_F-converges-to-M_f*}, where the ordering is the usual partial order of set containment on the finite sets $F$. However, this will require some futher discussion of \bV from Definition~\ref{eqn:bV}.

\begin{defn}\label{def:P_F}
  Given a finite subset $F \ci X$, define $P_F$ to be the projection to the subspace spanned by $\{v_x \suth x \in F\}$. 
\end{defn}

The purpose of $J$ in the following lemma is that it serves to intertwine $M_f$ with a more computable operator, see \eqref{eqn:T=JMJ} in the corollary below, and also \eqref{eqn:commutative-square}. Recall that \bV is defined in Definition~\ref{def:V} and discussed in Remark~\ref{rem:polarization}.

\begin{lemma}\label{thm:V-and-bV}
  A unitary equivalence between \bV and \HE is given by the operator  
  \linenopax
  \begin{align}\label{eqn:J}
    J:\HE \to \bV
    \qq\text{by}\qq
    Ju = V^{1/2} \gx, \q \text{ for } u = \sum_{x \in X} \gx_x v_x,
  \end{align}
  where convergence of the sum in \eqref{eqn:J} is with respect to \energy.
  \begin{proof}
  Let $u,w \in \spn\{v_x\}_{x \in F}$ be given by
  \linenopax
  \begin{align}\label{eqn:V-and-bV-der1}
    u = \sum_{x \in F} \gx_x v_x
    \qq\text{and}\qq
    w = \sum_{x \in F} \gh_x v_x,  
  \end{align}
  where $F$ is some finite subset of $X$. Then \eqref{eqn:V-prod} gives
  \linenopax
  \begin{align}\label{eqn:<u,w>=<gx,gh>}
    \la u,w\ra_\energy
    = \left\la \sum_{x \in F} \gx_x v_x, \sum_{y \in X} \gh_y v_y\right\ra_\energy
    = \sum_{x,y \in F} \cj{\gx_x} \gh_y \la v_x, v_y\ra_\energy
    = \la \gx,\gh\ra_\bV.
  \end{align}
  Now for general $u,w \in \HE$, let $P_F$ be as in Definition~\ref{def:P_F}, and compute
  \linenopax
  \begin{align*}
    \la u,w\ra_\energy
    &= \lim_{F \to X} \la P_F u, P_F w\ra_\energy
    = \lim_{F \to X} \la V_{F}^{1/2} \gx, V_{F}^{1/2} \gh \ra_\bV
    = \la V^{1/2} \gx, V^{1/2} \gh \ra_\bV,
  \end{align*}
  where the middle equality comes by \eqref{eqn:<u,w>=<gx,gh>}.
  \end{proof}
\end{lemma}

\begin{cor}\label{thm:T_F-converges-to-M_f*}
  Let $T_F$ be defined as in \eqref{eqn:boundedness-criterion}, and let $J$ be defined as in \eqref{eqn:J}.
  In the case when the equivalent conditions of Theorem~\ref{thm:boundedness-criterion} are satisfied, one has
  \linenopax
  \begin{align}\label{eqn:T=JMJ}
     T = \lim_{F \to X} T_F = J M_f^\ad J^\ad,
     \qq\text{and}\qq
     T^\ad = \lim_{F \to X} T_F^\ad = J M_f J^\ad,
  \end{align}
  where $M_f^\ad$ is the adjoint with respect to \energy, $T^\ad$ is the adjoint with respect to \bV, and the limit is taken in the strong operator topology.
  Thus, $M_f^\ad \cong \lim_{F \to X} T_F$.
  \begin{proof}        
    To see \eqref{eqn:T=JMJ}, first pick a finite $F \ci X$ and with $P_F$ as in Definition~\ref{def:P_F}, 
    \linenopax
    \begin{align*}
      \|M_f^\ad P_F u\|_\energy^2
      = \la P_F u, M_f M_f^\ad P_F u \ra_\energy 
      &= \sum_{x,y \in F} \cj{\gx_x} \gx_y f(x) \cj{f(y)} V_{xy} \\
      &= \la \gx, \cj{D_F} V_F D_F \gx\ra_\bV \\
      &= \| V_F^{1/2} D_F \gx\|_\bV^2.
    \end{align*}
    However, \eqref{eqn:boundedness-criterion} means that $T_F V_F^{1/2} = V_F^{1/2} D_F$, and so the computation continues with
    \linenopax
    \begin{align*}
      \|M_f^\ad P_F u\|_\energy^2
      = \| V_F^{1/2} D_F \gx\|_\bV^2
      = \| T_F V_F^{1/2} \gx\|_\bV^2.
    \end{align*}
    Now let $F \to X$ on both sides, and the proof follows by Theorem~\ref{thm:boundedness-criterion}.
  \end{proof}
\end{cor}

    Consequently, one has a commutative square
    \linenopax
    \begin{align}\label{eqn:commutative-square}
      \xymatrix{
        \HE \ar[r]^{M_f^\ad} \ar[d]_J & \HE \ar[d]^J \\
        \bV \ar[r]^T & \bV  
      }
    \end{align}

In light of Remark~\ref{rem:hardness-of-boundedness-criterion}, it will be helpful to have a condition which is only sufficient to ensure the boundedness of $M_f$ (not necessary), but is much easier to check.

\begin{theorem}\label{thm:boundedness-sufficiency}
  The operator $M=M_f$ satisfies 
  \linenopax
  \begin{align}\label{eqn:boundedness-sufficiency}
    \|M_f\|_{\HE \to \HE}
    \leq \sum_{x \in X} |f(x)| \sqrt{\cond(x) R(x)}
    = \sum_{x \in X} \frac{|f(x)|}{\sqrt{\prob[x \to o]}},
  \end{align}
  and is hence a bounded operator on \HE whenever the right side of \eqref{eqn:boundedness-sufficiency} converges.
  \begin{proof}
    For $F \ci X$ finite, let $f|_F = f\charfn{F}$ be the restriction of $f$ to $F$. Then Theorem~\ref{thm:d_x-is-bounded} and \eqref{eqn:M^ad-as-projn} give
    \linenopax
    \begin{align}\label{eqn:M_f|_F=sum(M_x)}
      M_{f|_F} = \sum_{x \in F} f(x) M_x
      = \sum_{x \in F} f(x) |\gd_x \ra \la v_x|,
    \end{align}
    where the summation is finite, so that $M_{f|_F}$ is clearly bounded. Now we show that $M_{f|_F}$ converges to $M_f$ in norm, as $F \to X$. Since
    \linenopax
    \begin{align*}
      \|M_{f|_F} \|_{\HE \to \HE}
      &\leq \sum_{x \in F} f(x) \left\||\gd_x \ra \vstr[2.1]\la v_x| \right\|_{\HE \to \HE} 
      = \sum_{x \in F} f(x) \|\gd_x\|_\energy \|v_x\|_\energy 
      = \sum_{x \in F} f(x) \sqrt{\cond(x) R(x)},
    \end{align*}
    we have \eqref{eqn:boundedness-sufficiency}. Moreover, when the right side of\eqref{eqn:boundedness-sufficiency} converges, then for any $\ge > 0$ there exists an $F_0$ such that 
    \linenopax
    \begin{align*}
      \sum_{x \in X \less F_0} f(x) \sqrt{\cond(x) R(x)} < \ge,
    \end{align*}
    which shows that $\lim_{F \to X} \|M_{f|_F} - M_f\|_{\HE \to \HE} = 0$,
    and \eqref{eqn:c(x)R(x)=Prob[x->o]} completes the proof.
  \end{proof}
\end{theorem}

One result appearing in the proof of Theorem~\ref{thm:boundedness-sufficiency} will be helpful on its own.

\begin{cor}\label{thm:convergence-of-M_f}
  If $M_f$ satisfies \eqref{eqn:boundedness-sufficiency}, then $M_{f|_F}$ converges to $M_f$ in norm, where $M_{f|_F}$ is as in \eqref{eqn:M_f|_F=sum(M_x)}. In particular, \eqref{eqn:boundedness-sufficiency} implies
  \linenopax
  \begin{align}\label{eqn:M_f=sum(M_x)}
    M_{f} = \sum_{x \in x} f(x) M_x
    = \sum_{x \in X} f(x) |\gd_x \ra \la v_x|,
  \end{align}
  where the sum converges in the norm operator topology.
\end{cor}

It seems doubtful that $M_{f|_F}$ converges to $M_f$ in norm, in general. However, we do have a partial result in this direction, in Theorem~\ref{thm:M_f-as-strong-limit-of-restrictions}.

\begin{lemma}\label{thm:M_f-as-strong-limit-of-projsandwich}
  Let $\{F_n\}_{n=1}^\iy$ be an exhaustion of $X$, and define $P_n$ to be the projection to \smash{$\spn\{v_x \suth x \in F_n\}$}, for each $n \in \bN$. 
  If $M_f$ is bounded, then for each $n$, there is an $m=m_n$ with 
  \linenopax
  \begin{align}\label{eqn:P_nM_fP_n-truncated}
    P_n M_{f} P_n = P_n M_{f_{m_n}} P_n,
  \end{align}
  where $f_k = f|_{F_k} = f\charfn{F_k}$ is the restriction of $f$ to $F_k$.
  \begin{proof}
    Since the energy kernel has dense span in \HE, we can 
    apply the Gram-Schmidt algorithm to obtain an onb $\{\onb_x\}_{x \in X}$.%
    \footnote{This is carried out in more detail in \cite[\S3.1]{bdG}. Note that $v_o = v_{x_0}$ is \emph{not} included in the enumeration.} 
    Thus we can write
    \linenopax
    \begin{align}
      P_n M_f P_n 
      &= \sum_{x} \sum_{y \leq x} \sum_{z \leq y} f(x) 
        |\onb_y\ra \la \onb_y| \, |\gd_x\ra \la v_x| \, |\onb_z\ra \la \onb_z| \notag \\
      &= \sum_{x} \sum_{y \leq x} \sum_{z \leq y} f(x) 
        \la \onb_y, \gd_x\ra_\energy \la v_x, \onb_z\ra_\energy  |\onb_y\ra  \la \onb_z| .
        \label{eqn:M_f-as-strong-limit-of-projsandwich-der1}
    \end{align}
    However, for all $n$, there exists an $m \geq n$ (which we write as $m_n$ to emphasize the dependence on $n$) such that, for $x \in F_n$ and $y,z \in F_{m_n}^C$, one has 
    \linenopax
    \begin{align*}
      \la \onb_y, \gd_x\ra_\energy = \la v_x, \onb_z\ra_\energy = 0.
    \end{align*}
    This essentially follows from the finite range of \cond and the nature of the Gram-Schmidt algorithm and shows that the sum in \eqref{eqn:M_f-as-strong-limit-of-projsandwich-der1} is finite. 
  \end{proof}
\end{lemma}

\begin{theorem}\label{thm:M_f-as-strong-limit-of-restrictions}
  Let $\{F_n\}_{n=1}^\iy$ be an exhaustion of $X$, and for a fixed $f:X\to\bC$, and let $f_n = f|_{F_n} = f\charfn{F_n}$ be the restriction of $f$ to $F_n$. 
  If $M_f$ is bounded, then $P_n M_{f_{m_n}} P_n$ converges to $M_f$ in the strong operator topology, where $M_{f_{m_n}}$ is a finite-dimensional suboperator of $M_f$ as in Lemma~\ref{thm:M_f-as-strong-limit-of-projsandwich}.
  \begin{proof}
    Note that $P_n M P_n$ converges strongly to $M$ whenever $M$ is a bounded operator, by general operator theory. Then by Lemma~\ref{thm:M_f-as-strong-limit-of-projsandwich}, the right side of \eqref{eqn:P_nM_fP_n-truncated} also converges to $M_f$ in the strong operator topology.
  \end{proof}
\end{theorem}

\begin{cor}\label{thm:ran(M_f)-in-Fin}
  If $M_f$ is bounded, then the range of $M_f$ lies in \Fin.
  \begin{proof}
    Since $M_x = |\gd_x\ra \la v_x|$ by \eqref{eqn:M^ad-as-projn}, and $\ran M_x \ci \bC\gd_x$, this follows immediately from Theorem~\ref{thm:M_f-as-strong-limit-of-restrictions}.    
  \end{proof}
\end{cor}


\section{Bounded functions of finite energy}
\label{sec:Bounded-functions-of-finite-energy}

In the preceding section, we considered the functions $f$ for which $M_f$ is a bounded operator. In this section, we consider the algebra of bounded functions $f$ in \HE. 
Neither of these spaces of operators is contained in the other, as illustrated in the examples of \S\ref{sec:Examples}.

\begin{defn}\label{def:bounded-in-HE}
  For $u \in \HE$, denote $\|u\|_\iy := \sup_{x \in \Graph} |u(x)-u(o)|$, and say $u$ is \emph{bounded} iff $\|u\|_\iy < \iy$.
\end{defn}

In \cite{bdG}, we give two ways of constructing a Gel'fand triple $\Schw \ci \HE \ci \Schw'$ for the energy space. Here \Schw is a dense subspace of \sH which should be thought of as a space of test functions. Indeed, \Schw is equipped with a strictly finer ``test function topology'' given by a countable system of seminorms; this yields a Fr\'echet topology which is strictly finer than the norm topology on \HE. Then $\Schw'$ is the dual of \Schw with respect to this finer (Fr\'echet) topology, so that one obtains a strict containment $\sH \subsetneq \Schw'$. In fact, it is possible to chose the seminorms in such a way that the inclusion map of \Schw into \sH is a {nuclear} operator.

To make all this concrete, let us briefly describe the two constructions given in \cite{bdG}.
\begin{enumerate}
  \item Fix an enumeration of the vertices, and apply the Gram-Schmidt procedure (as in the proof of Lemma~\ref{thm:M_f-as-strong-limit-of-projsandwich}) to $\{v_{x_n}\}_{n=1}^\iy$ to obtain an orthonormal basis $\{\onb_n\}_{n=1}^\iy$. Then define \smash{$\Schw = \bigcap_{p \in \bN} \{s \suth \|s\|_p < \iy\}$}, where the Fr\'{e}chet \emph{$p$-seminorm} of $s = \sum_{n \in \bN} s_n \onb_n$ is given by
  \linenopax
  \begin{align}\label{eqn:p-seminorm}
     \|s\|_p 
     := \left(\sum_{n \in \bN} n^p |s_n|^2\right)^{1/2},
      \qq s \in \Schw, p \in \bN.
  \end{align}
  \item In the case where \Lap is an unbounded operator on \HE, let \LapS be a self-adjoint extension of \Lap and define $\Schw := \dom(\LapS^\iy) = \bigcap_{p=1}^\iy \dom(\LapS^p)$, with Fr\'echet $p$-seminorms $\|u\|_p := \| \LapS^p u\|_\energy$. (Details regarding the precise domain of \Lap and \LapS in this context may be found in \cite{bdG}.)
\end{enumerate}

Either way, it turns out that $\Schw'$ is large enough to support a nice probability measure, even though \sH is not. This allows one to establish an isometric embedding of \HE into the Hilbert space $L^2(\Schw',\prob)$, where \prob is a Gaussian probability measure on $\Schw'$.

\begin{theorem}[Wiener embedding, {\cite[Thm.~5.2]{bdG}}]
  \label{thm:HE-isom-to-L2(S',P)}
  The Wiener transform $\sW:\HE \to L^2(\Schw',\prob)$ defined by
  \linenopax
  \begin{equation}\label{eqn:Gaussian-transform}
    \sW : v \mapsto \tilde v,
    \q \tilde v(\gx) := \la v, \gx\ra_\Wiener,
  \end{equation}
  is an isometry. The extended reproducing kernel $\{\tilde v_x\}_{x \in \Graph}$ is a system of Gaussian random variables from which one can obtain the free effective resistance (see Definition~\ref{def:R(x)}) by
  \linenopax
  \begin{equation}\label{eqn:R(x,y)-as-expectation}
    R^F(x,y) = \energy(v_x-v_y) = \Ex_\gx((\tilde v_x - \tilde v_y)^2).
  \end{equation}
  Moreover, for any $u,v \in \HE$, the energy inner product extends directly as
  \linenopax
  \begin{equation}\label{eqn:Expectation-formula-for-energy-prod}
    \la u, v \ra_\energy
    = \Ex_\gx\left( \cj{\tilde{u}} \tilde{v} \right)
    = \int_{\Schw'} \cj{\tilde{u}} \tilde{v} \,d\prob.
  \end{equation}
\end{theorem}

\begin{remark}\label{rem:HE-into-L2(S',prob)-gives-hermitian-multiplication}
  The Wiener transform gives a representation of the Hilbert space \HE as an $L^2$ space of functions on a probability ``sample space'' $(\Schw',\prob)$. This is useful in many ways. 
  \begin{enumerate}[(a)]
    \item While direct computation in \HE is typically difficult (when solving equations, for example), passing to the transform allows us instead to convert
geometric problems in \HE into manipulation of functions on $\Schw'$ or on a
subspace of it. 
  \item As we show in this section, problems involving bounded operators in \HE can
be subtle. The Wiener transform immediately offers a maximal abelian algebra
of bounded operators, viz., multiplication by $L^\iy$ functions on $\Schw'$.%
  \footnote{The multiplication operator on \HE ``before the transform'' (discussed in \S\ref{sec:Bounded-multiplication-operators}--\S\ref{sec:Algebras-of-multiplication-operators}) should not be confused with those in $L^2(\Schw',\prob)$ ``after the transform''.}
  \end{enumerate}
\end{remark}

\begin{defn}\label{defn:algE}
  Denote the collection of bounded functions of finite energy by
  \linenopax
  \begin{equation}
    \algE := \{u \in \HE \suth u \text{ is bounded}\}.
  \end{equation}
  Define multiplication on $\algE$ by the pointwise product
  \linenopax
  \begin{equation}
    (u_1 u_2)(x) := u_1(x) u_2(x),
  \end{equation}
  and a norm on $\algE$ by
  \linenopax
  \begin{equation}\label{eqn:algE-norm}
    \|u\|_\sA := \|u\|_\iy + \|u\|_\energy.
  \end{equation}
\end{defn}

\begin{lemma}\label{defn:algE-is-a-Banach-alg}
  $(\algE,\| \cdot \|_\sA)$ is a Banach algebra.
  \begin{proof}
    It is obvious that $u_1 u_2$ is bounded; one checks that $u_1 u_2 \in \dom \energy$ by directly computing,
    \linenopax
    \begin{align*}
      \|u_1 u_2\|_\energy^2
      &= \tfrac12 \sum_{x,y} \cond_{xy} |u_1 u_2(x) - u_1 u_2(y)|^2 \\
      &= \tfrac12 \sum_{x,y} \cond_{xy} |(u_1(x) - u_1(y)) u_2(x) + (u_2(x) - u_2(y)) u_1(x)|^2 \\
      &\leq \tfrac12 \sum_{x,y} \cond_{xy} (|(u_1(x) - u_1(y)| |u_2(x)| + |u_2(x) - u_2(y)| |u_1(x)|)^2 \\
      &= \tfrac12 \sum_{x,y} \cond_{xy} (|(u_1(x) - u_1(y)|^2 |u_2(x)|^2 + \tfrac12 \sum_{x,y} \cond_{xy} |u_2(x) - u_2(y)|^2 |u_1(x)|^2 \\
      &\hstr[10] + \tfrac12 \sum_{x,y} \cond_{xy} |u_1(x)| |u_2(x)| |u_1(x) - u_1(y)| |u_2(x) - u_2(y)| \\
      &\leq \|u_2^2\|_\iy \|u_1\|_\energy^2 +  2\|u_1\|_\iy \|u_2\|_\iy |\la u_1, u_2\ra_\energy| + \|u_1^2\|_\iy \|u_2\|_\energy^2,
    \end{align*}
    which is clearly finite. This estimate also implies that $(u,v) \mapsto \| uv \|_\sA$ is a closed linear functional on the product space $\algE \times \algE$. The closed graph theorem then implies that it is continuous, i.e.,
    \linenopax
    \begin{align*}
      \| uv \|_\sA \leq C \|u\|_\sA \|v\|_\sA,
      \q\text{for all } u, v \in \algE. 
    \end{align*}
    It is a standard argument that one can then find an equivalent norm for which the same inequality holds with $C=1$; see \cite{KadisonRingroseI}, for example.
  \end{proof}
\end{lemma}

\begin{defn}\label{defn:Gelfand-space}
  By the \emph{Gel'fand space} of a Banach algebra \sA, we mean the spectrum $\spec(\sA)$ realized as either the collection of maximal ideals of \sA or as the collection of multiplicative linear functionals on \sA. See \cite{Arveson:spectral-theory,Arveson:invitation-to-Cstar}.

  Let $\gz \in \spec(\algE)$ denote a multiplicative linear functional on $\algE$, so that $\ker \gz$ is a maximal ideal of $\algE$, and let $\gF_\sA: \algE \to C(\spec(\algE))$ denote the Gel'fand transform, so that $\gF_\sA(v)(\gz) := \gz(v)$. 
\end{defn}

There is a norm equivalent to the one given in \eqref{eqn:algE-norm} with respect to which \algE becomes a Banach algebra (see \cite{KadisonRingroseI}, e.g.), and we are concerned with the Gel'fand space of this one.

\begin{lemma}\label{defn:algE=C(spec(algE))}
  As a Banach algebra, $\algE$ is isometrically isomorphic to $C(\spec(\algE))$.
  \begin{proof}
    We need to show that $\ker \gF_\sA = 0$. This is equivalent to showing that $\algE$ is \emph{semisimple}, i.e., that the intersection of all the maximal ideals is $0$. It therefore suffices to show that an intersection of a subcollection of the maximal ideals is $0$. Let $L_x$ denote the multiplicative linear functional defined by $L_x u := u(x)$. Since $L_x u = \la u, v_x \ra_\energy$ under convention \eqref{eqn:o-convention}, and $\{v_x\}$ is dense in \HE (and therefore total), it follows that $\bigcap \ker L_x = 0$.
  \end{proof}
\end{lemma}

\begin{defn}\label{defn:algF}
  Recall from Definition~\ref{def:Fin} that $\spn\{\gd_x\}$ is the collection of functions of finite support; see also the first paragraph of \S\ref{sec:Bounded-multiplication-operators}. If we complete $\spn\{\gd_x\}$ in the sup norm, we obtain the collection of bounded functions on \Graph, and if we complete in \energy, we obtain \Fin. Therefore, the closure of $\spn\{\gd_x\}$ in the norm of \algE is 
  \linenopax
  \begin{align}\label{eqn:algF}
     \algF := \Fin \cap \algE.
  \end{align}
\end{defn}

\begin{lemma}\label{thm:algF-is-closed}
  \algF is a closed ideal in \algE.
  \begin{proof}
    Fix $x \in \Graph$ and let $\gd_x \in \algF$ be the characteristic function of $\{x\}$ as defined in Definition~\ref{def:Fin}. Take any finite set $F \ci \Graph$ and any linear combination $f = \sum_{x \in F} \gx_x \gd_x$. Since $v \cdot \gd_x = v(x) \gd_x$, one has $v \cdot f = \sum_{x \in F} \gx_x v(x) \gd_x$, which is clearly supported in $F$ again. This shows that the collection of all finitely supported functions on \Graph is an ideal. 
    
    Now for $f \in \algF$, take $\{f_n\}$ where each $f_n$ has finite support and $\|f-f_n\|_\sA \to 0$. This is possible in view of Definition~\ref{defn:algF}. Since $v \cdot f_n \in \spn\{\gd_x\}$ by the first part,
    \linenopax
    \begin{equation}
      \| (v \cdot f) - (v \cdot f_n)\|_\sA 
      = \| v \cdot (f - f_n)\|_\sA
      \leq \|v\|_\sA \|f-f_n\|_\sA \to 0,
    \end{equation}
    shows $v \cdot f \in \algF$ (by Definition~\ref{defn:algF} again). 
  \end{proof}
\end{lemma}

\begin{defn}\label{defn:algH}
  Since $\algF$ is a closed ideal, it is standard that 
  \linenopax
  \begin{equation}
    \algH := \algE/\algF
  \end{equation}
  is an algebra, and in fact a Banach algebra under the usual norm
  \begin{equation}
    \| [u] \|_\Harm := \inf \{\|u+f\|_\sA \suth f \in \sB_{\Fin}\}.
  \end{equation}
\end{defn}

\begin{theorem}\label{thm:Gelfand-space-is-bdG}
  The Gel'fand space of \algH contains $\bd \Graph$, and there is an isometric embedding $\algH \hookrightarrow C(\bd \Graph)$.
  \begin{proof}
    Recall from \cite[Cor.~4.5]{bdG} that for $v_x \in \HE$, one defines $\tilde v_x \in L^2(\Schw',\prob)$ by
    \linenopax
    \begin{equation}
      \tilde v_x(\gx) = \la v_x, \gx\ra_\sG = \lim_{n \to \iy} \la v_{x,n}, \gx \ra_\energy, 
    \end{equation}
    where $\{v_{x,n}\}_{n \in \bN}$ is any sequence in \Schw converging to $v_x$, and that with this extension, harmonic functions in \HE have the boundary representation
    \linenopax
    \begin{equation}\label{eqn:integral-repn-of-h}
      h(x) = \int_{\Schw'} \tilde v_x(\gx)\dn{\tilde{h}}(\gx) \, d\prob(\gx) + h(o).
    \end{equation}
    (Full details on this notation may be found in \cite{bdG} or \cite{OTERN}.)
    It follows immediately from this representation that if $\tilde h = 0$ on $\bd \Graph$, then $h=0$ everywhere on \Graph. (Here, $\tilde h = 0$ on $\bd \Graph$ means $\lim_{n \to \iy} h(x_n)=k$ for some $k \in \bC$ and any sequence $\{x_n\}$ with $\lim_{n \to \iy} x_n = \iy$.)
    
    Given any $\gb \in \bd\Graph$, the evaluation $\gc_\gb(u) := \tilde u(\gb)$ defines a multiplicative linear functional on $\algH$, so that $\bd \Graph$ is contained in the Gel'fand space of \algH. 
  \end{proof}
\end{theorem}

\begin{theorem}\label{thm:Harm=0-means-Gelfand-is-1pt}
  If $\Harm = 0$, then the Gel'fand space of \algE is $\Graph \cup \{\iy\}$.
  \begin{proof}
    %
    Let $\gc \in \spec(\algE)$ and apply it to both sides of $v \cdot \gd_x = v(x) \gd_x$ (the left side is a product in \algE and the right side is a scalar multiple of $\gd_x$) to obtain $\gc(v) \cdot \gc(\gd_x) = v(x) \gc(\gd_x)$, and hence
    \linenopax
    \begin{equation}\label{eqn:Gelf-chi-split}
      \gc(\gd_x) \cdot (\gc(v) - v(x)) = 0, \q \forall x \in \Graph, \forall v \in \algE.
    \end{equation}
    This implies (i) $\gc(\gd_x) = 0$ for all $x$, or else (ii) $\exists y \in \Graph$ for which $\gc(\gd_y) \neq 0$. Since $\Harm = 0$, Theorem~\ref{thm:HE=Fin+Harm} implies that \gc is determined by its action on $\{\gd_x\}_{x \in \Graph}$. Thus, only the zero functional satisfies $\gc(\gd_x) = 0$ for all $x \in \Graph$, and we may safely ignore case (i). For case (ii), it follows that $\gc(\gd_x)=0$ for all $x \neq y$, so $\gc(v) = v(y)$ by \eqref{eqn:Gelf-chi-split}. This shows that \gc corresponds to evaluation at the vertex $y$; note that the uniqueness of $y$ for which $\gc(\gd_y) \neq 0$ is implicit. 
    
    Observe that $C(\Graph)$ is not unital, because the constant function $\one \simeq 0$ in \HE. We unitalize \algE in the usual way: 
    \linenopax
    \begin{equation}
      \tilde{{\algE}} = \algE \times \bC
      \q\text{with}\q
      (a_1,\gl_1) (a_1,\gl_1) := (a_1 a_2 + \gl_2a_1 + \gl_1a_2, \gl_1 \gl_2).
    \end{equation}
    The unit in this new algebra is then $(0,1)$. By standard theory, this corresponds to taking the one-point compactification of \Graph. 
  \end{proof}
\end{theorem}
Roughly speaking, taking the one-point compactification of \Graph corresponds to conjoining the single multiplicative linear functional ``evaluation at \iy'' to \algE. It is known from \cite{Ancona} that when $\Harm=0$, $u(x)$ tends to a common value along \prob-a.e. path to \iy, for any $u \in \HE$. 

\begin{conj}
  We conjecture that the converses of Theorem~\ref{thm:Gelfand-space-is-bdG} and of Theorem~\ref{thm:Harm=0-means-Gelfand-is-1pt} both hold. In other words, we expect that $\algH \cong C(\bd \Graph)$, and that if $\Harm \neq 0$, then the Gel'fand space of \algE contains at least two elements that don't correspond to any vertex of \Graph.
\end{conj}


\section{Examples}
\label{sec:Examples}

\begin{ex}
  The following example shows that even though $v_x$ is a bounded function on any network (Lemma~\ref{thm:monopoles-and-dipoles-are-bounded}), the corresponding multiplication operator may not be bounded. This highlights the disparity between $C^\ad(\HE)$ from Definition~\ref{def:multiplier-C*algebra} and \algE from Definition~\ref{defn:algE}. 
  
  \begin{figure}
    \scalebox{0.80}{\includegraphics{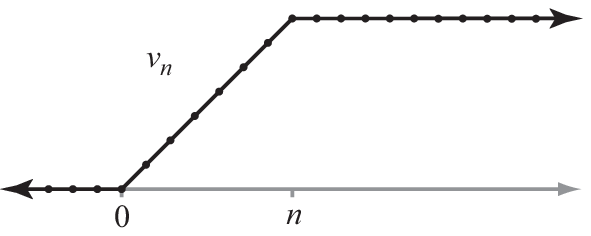}}
    \centering
    \caption{\captionsize The energy kernel element $v_n$ on the integer network $(\bZ,\one)$.}    
    \label{fig:vx-in-Z1}
  \end{figure}
  
  Consider the integer network with unit conductances $(\bZ,\one)$: 
  \linenopax
  \begin{align*}
    \xymatrix@W=-2pt{
      \dots \ar@{-}[r]^1
      & \vertexdot{-2} \ar@{-}[rr]^1
      && \vertexdot{-1} \ar@{-}[rr]^1
      && \vertexdot{0} \ar@{-}[rr]^1
      && \vertexdot{1} \ar@{-}[rr]^1
      && \vertexdot{2} \ar@{-}[rr]^1
      && \vertexdot{3} \ar@{-}[r]^1
      & \dots
    }
  \end{align*} 
  We label the vertex $x_n$ by `$n$' to simplify notation. Then if \eqref{eqn:boundedness-sufficiency} held, Corollary~\ref{thm:convergence-of-M_f} would give
  \linenopax
  \begin{align*}
    M_{v_n}
    &= |\gd_1\ra \la v_1| + 2|\gd_2\ra \la v_2| + \dots 
      + n|\gd_n\ra \la v_n| + n|\gd_{n+1}\ra \la v_{n+1}| + \dots,
  \end{align*}
  for each fixed $n$. The operator norm corresponding to one of these terms is
  \linenopax
  \begin{align*}
    \left\| n \vstr[2.2]|\gd_{n+k}\ra \la v_{n+k}|\right\| 
    = n\|\gd_{n+k}\|_\energy \|v_{n+k}\|_\energy
    = n \sqrt2 \sqrt{n+k} 
    \limas{k} \iy,
  \end{align*}
  so clearly \eqref{eqn:boundedness-sufficiency} cannot hold. 
  
  Checking Theorem~\ref{thm:boundedness-criterion} directly is harder; one must compute
  \linenopax
  \begin{align*}
   \|M_{v_n}\|_{\HE \to \HE}
    = \sup_F \|V_F^{1/2} D_F V_F^{-1/2}\|_{\ell^2 \to \ell^2},
  \end{align*}
  where the latter is the operator norm on $\ell^2(F)$ and $F$ ranges over all finite subsets of $X$. For our purposes, it will suffice to consider sets $F$ of the form $F=\{1,2,\dots,n\}$. The matrix for $V_F$ is then
  \linenopax
  \begin{align*}
    \left[\begin{array}{ccccccccc}
    1 & 1 & 1 & 1 & \dots \\
    1 & 2 & 2 & 2 & \dots \\
    1 & 2 & 3 & 3 & \dots \\
    1 & 2 & 3 & 4 & \dots \\
    \vdots & \vdots & \vdots & \vdots & \ddots & \vdots & \vdots & \vdots \\
    &&&& \dots & n-2 & n-2 & n-2 \\
    &&&& \dots & n-2 & n-1 & n-1 \\
    &&&& \dots & n-2 & n-1 & n \\
    \end{array}\right],\\
  \end{align*}
  but $V^{-1/2}$ is a complicated even for small $F$. For example, for $V_F = \left[\begin{smallmatrix} 1 & 1 & 1 \\ 1 & 2 & 2 \\ 1 & 2 & 3 \end{smallmatrix}\right]$, one has
  \scalebox{0.80}{\begin{minipage}{5in}
  \begin{align*}
    \hstr[6]
    \scalebox{1.25}{$V_F^{1/2} = $}
    \left[\begin{array}{ccc}
 \llbracket -8-28 \gg+49 \gg^3, 3\rrbracket  & \llbracket 8-28 \gg+49 \gg^3, 2\rrbracket  & \llbracket 1+7 \gg-49 \gg^2+49 \gg^3, 2\rrbracket  \\
 \llbracket 8-28 \gg+49 \gg^3, 2\rrbracket  & \llbracket 13-21 \gg-49 \gg^2+49 \gg^3, 3\rrbracket  & 1+\llbracket -8+98 \gg^2+49 \gg^3, 2\rrbracket  \\
 \llbracket 1+7 \gg-49 \gg^2+49 \gg^3, 2\rrbracket  & \llbracket 41-49 \gg-49 \gg^2+49 \gg^3, 2\rrbracket  & \llbracket 97-105 \gg-49 \gg^2+49 \gg^3, 3\rrbracket 
\end{array}\right],
  \end{align*}
  \end{minipage}}\\
  where $\llbracket p(\gg), k\rrbracket$ is the root of the polynomial $p(\gg)$ closest to the number $k$.
\end{ex}

\subsection*{Acknowledgements}

The authors are grateful for stimulating comments, helpful advice, and valuable references from Daniel Alpay, Ilwoo Cho, Keri Kornelson, Paul Muhly, Anna Paolucci, Myung-Sin Song, Karen Shuman, and others. We also thank the students and colleagues who have endured our talks on this material and raised fruitful questions. The work of the first named author was completed while visiting at the Ben Gurion University (Israel), and PJ is grateful to Daniel Alpay and BGU for support and hospitality.

\bibliographystyle{alpha}
\bibliography{networks}

\pgap[0.5cm]

\end{document}